\documentclass[smallextended]{svjour3}       
\RequirePackage{fix-cm}

\smartqed  
\usepackage{graphicx}
\usepackage{subfigure}
\usepackage{tikz}
\usetikzlibrary{calc}
\usepackage{amssymb}
\usepackage{amsmath}

\newcommand{\vertiii}[1]{{\left\vert\kern-0.25ex\left\vert\kern-0.25ex\left\vert #1 
    \right\vert\kern-0.25ex\right\vert\kern-0.25ex\right\vert}}
 \journalname{}
\begin{document}

\title{Approximation approach to the fractional BVP with the Dirichlet type boundary conditions
}

\titlerunning{Approximation approach to the FBVP...}        

\author{Kateryna Marynets* \and Dona Pantova        
}

\authorrunning{K. Marynets \and D. Pantova} 

\institute{\textbf{Corresponding author*:}\\ Kateryna Marynets \at
              Delft Institute of Applied Mathematics, Faculty of Electrical Engineering, Mathematics and Computer  Science,  Delft  University  of  Technology,  Mekelweg 4, 2628 CD Delft,  The~Netherlands  \\
\email{K.Marynets@tudelft.nl}           
     \and Dona Pantova \at
              Delft Institute of Applied Mathematics, Faculty of Electrical Engineering, Mathematics and Computer  Science, Delft  University  of  Technology,  Mekelweg 4, 2628 CD Delft,  The~Netherlands \\
             \email{D.H.Pantova@tudelft.nl}           
           }
\date{Received: date / Accepted: date}

\maketitle

\begin{abstract}
We use a numerical-analytic technique to construct a sequence of successive approximations to the solution of a system of fractional differential equations, subject to Dirichlet boundary conditions. We prove the uniform convergence of the sequence of approximations to a limit function, which is the unique solution to the boundary value problem under consideration, and give necessary and sufficient conditions for the existence of solutions. The obtained theoretical results are confirmed by a model example.

\keywords{Fractional differential equations \and Dirichlet boundary conditions \and Approximation of solutions \and Brouwer degree}
\subclass{34A08 \and 34K07 \and 34K28}
\end{abstract}

\section{Introduction}\label{sec1}

\,\,\,\,\, The topic of fractional differential equations (FDEs) has become an active area of research over the past several decades. The study of existence and uniqueness of solutions, and of the evolution of systems described by FDEs is of theoretical, as well as practical interest to mathematicians and scientists who aim to model the behaviour of complex dynamical systems. The main advantage of fractional calculus operators is in their ability to capture non-local and long-term memory effects \cite{podbulny,Herrmann}. This property allows the development of more realistic models using FDEs for complex phenomena, such as anomalous diffusion, the behaviour of viscoelastic materials, transport properties, and fluid flows \cite{kilbas}. Generally, real-world processes are non-linear, and thus described by FDEs containing non-linearities, the exact solutions to which are often not possible to obtain. This has motivated the development of approximate solution methods, such as the numerical-analytic method, which combines deriving an approximate solution in analytic form with the numerical calculation of the parameters describing the solution's behaviour.

In this paper, we apply a numerical-analytic technique, which was originally developed for approximating the solutions to periodic boundary value problems (BVPs) for ordinary differential equations \cite{ronsam}, and later on adapted for FDEs (see e.g. \cite{Marynets2}), to the study of the solvability and constructive approximation of solutions for systems of FDEs of the Caputo type with Dirichlet boundary conditions. We derive integral equations which give the solution to the initial value problem (IVP), corresponding to the original BVP, and construct a sequence of functions, depending on a vector-parameter, which is found as a root of the so-called determining system of algebraic equations. We prove the uniform convergence of the sequence of functions to a limit function, and show the relationship between the limit function and the original BVP. Finally, we prove two results on the necessary and sufficient conditions for the existence of solutions of the BVP. 

The obtained theoretical results and the effectiveness of the developed technique are confirmed on an example of the gyre equation for the Antarctic Circumpolar Current considered in the fractional setting (for more details about the mathematical model of the Antarctic Circumpolar Current we refer to \cite{ConJohn,Q,marynets1}).

\section{Problem Setting}\label{sec2}

 In this paper, we consider a BVP for a FDS of the form 

\begin{equation}
\label{fds}
    {^C_0}{D}^{p}_{t} u(t) = f(t, u(t))
\end{equation}
for some $p \in (1,2]$, and subjected to the non-homogeneous Dirichlet boundary conditions 
\begin{align}
\label{bcs}
u(0) = \alpha_{1}, \,\,\, u(T) = \alpha_{2},
\end{align}
where ${^C_0}{D}^{p}_{t}$ is the Caputo fractional derivative (see \cite{podbulny}, Def. 2.138) with lower limit at $0$, $t \in [0,T]$, $u: [0, T] \rightarrow D$, $f: G \rightarrow \mathbb{R}^{n}$ are continous functions, $ G := [0, T] \times D$ and $D \subset \mathbb{R}^{n}$ is a closed and bounded domain. 

\medskip
We assume the function $f$ in system (\ref{fds}) to be bounded by a constant vector $M = col(M_{1}, M_{2}, ..., M_{n}) \in \mathbb{R}^{n}$ and to satisfy the Lipschitz condition with a non-negative real matrix $K = (k_{ij})_{i,j=1}^{n}$, i.e. the following inequalities 
\begin{align}
\label{bounded}
\lvert f(t, u(t)) \rvert \leq M,
\end{align}
\begin{equation}
\label{Lipschitz}
\lvert f(t,u_{1}) - f(t,u_{2}) \rvert \leq K \lvert u_{1}-u_{2} \rvert
\end{equation}
hold for $t \in [0, T]$, $u, u_{1}, u_{2} \in D$.

\medskip 
\indent \textit{Note that the operations $\lvert  \cdot \rvert$, $=$, $\leq$, $\max$, etc. between matrices and vectors are understood componentwise.}

\medskip 
\indent Suppose that the set 
\begin{align}
\label{Dbeta}
D_{\beta} := \{ \chi_{0} \in D : \{\lvert u - \chi_{0} \rvert \leq \beta, \,\, u \in \mathbb{R}^{n} \} \subset D \}
\end{align}
is non-empty, where 
\begin{align*}
    \chi_{0} = u(0),
\end{align*}
\begin{align}
\label{beta}
\beta = \frac{M T^{p}}{2^{2p-1}\Gamma(p+1)}, 
\end{align}
and the spectral radius $r(Q)$ of the matrix 
\begin{align}
\label{matrix}
Q :=  \frac{K T^{p}}{2^{2p-1}\Gamma(p+1)}
\end{align}
satisfies 
\begin{align}
\label{spectrad}
r(Q) < 1.
\end{align}
\textit{We aim to find a solution of the FDS (\ref{fds}) which satisfies the Dirichlet boundary conditions (\ref{bcs}) in the space of continuous functions $u: [0, T] \rightarrow D$.} 

\medskip
For this purpose, let us connect the BVP (\ref{fds}), (\ref{bcs}) to the following parametrized sequence of functions $\{u_{m}(\cdot, \chi_1)\}_{m \in \mathbb{Z}^{+}_{0}}$, $\mathbb{Z}^{+}_{0} = \{0, 1, 2, ...\}$, given by the iterative formula:
\begin{align}
\begin{split}
\label{seq2}
    u_{m}(t, \chi_{1}) :=& \alpha_{1} +\chi_{1}t+ (\alpha_{2}-\alpha_{1}-\chi_{1}T)\Big(\frac{t}{T}\Big)^{p}\\
    &+ \frac{1}{\Gamma(p)} \Big[ \int_{0}^{t}(t-s)^{p-1}f(s,u_{m-1}(s, \chi_{1}))ds\\ &-\Big(\frac{t}{T}\Big)^{p}\int_{0}^{T}(T-s)^{p-1}f(s,u_{m-1}(s, \chi_{1}))ds \Big],\\
    u_{0}(t, \chi_{1}) :=& \alpha_{1} +\chi_{1}t+ (\alpha_{2}-\alpha_{1}-\chi_{1}T)\Big(\frac{t}{T}\Big)^{p},
\end{split}
\end{align}
where $ t \in [0, T]$, $u_{0}(t, \chi_{1}) \in D$, and $\chi_{1} \in \Omega \subset \mathbb{R}$ is the value of the first derivative of $u(t)$ at $t=0$, i.e. $u'(0) = \chi_{1}$. Here $\Gamma(p)$ is the Gamma function.

\section{Convergence of functional sequences}\label{sec3}

\subsection{\textit{\textbf{Auxiliary Statements}}}\label{subsec2}

\begin{lemma}\label{lemma1}\cite{Marynets}
 If $f(t)$ is a continuous function on $t \in [a, b]$, then the following estimate 
\begin{equation}
\begin{gathered}
\label{estlem1}
    \frac{1}{\Gamma(p) } \Biggl \lvert \int_{a}^{t}(t-s)^{p-1}f(s)ds -\Big(\frac{t-a}{b-a}\Big)^{p}\int_{a}^{b}(b-s)^{p-1}f(s)ds \Biggr \rvert \\ \leq \alpha_{1}(t) \max_{a\leq t \leq b} \lvert f(t) \rvert,
\end{gathered}
\end{equation}
where 
\begin{align}
\label{alpha1t}
    \alpha_{1}(t)  := \frac{2(t-a)^{p}}{\Gamma(p+1)} \Big(\frac{b-t}{b-a}\Big)^{p},
\end{align}
holds for all $ t \in [t_{1}, T]$.\\
\end{lemma}
\begin{lemma}\label{lemma2}\cite{Marynets}
Let $\{ \alpha_{m} (\cdot)\}_{m \geq 1}$ be a sequence of continuous functions on ${t \in [a, b]}$, given by 
\begin{align*}
\alpha_{m}(t) :=& \frac{1}{\Gamma(p)} \Biggr[ \int_{a}^{t}\Big[(t-s)^{p-1}-\Big(\frac{t-a}{b-a}\Big)^{p}(b-s)^{p-1}\Big]\alpha_{m-1}(s)ds \\
&+\Big(\frac{t-a}{b-a}\Big)^{p}\int_{t}^{b}(b-s)^{p-1}\alpha_{m-1}(s)ds \Biggl],
\end{align*}
where 
\begin{align*}
    \alpha_{0}(t) &:= 1,\\
    \alpha_{1}(t) &:= \frac{2(t-a)^{p}}{\Gamma(p+1)} \Big(\frac{b-t}{b-a}\Big)^{p}.
\end{align*}
Then the estimate 
\begin{align}
\label{estlem2}
    \alpha_{m+1}(t) \leq \frac{(b-a)^{mp}\alpha_{1}(t)}{2^{[m(2p-1)]}[\Gamma(p+1)]^{m}} \leq \frac{(b-a)^{(m+1)p}}{2^{[(m+1)(2p-1)]}[\Gamma(p+1)]^{m+1}} 
\end{align}
holds for $m \in \mathbb{Z}^{+}_{0} $.\\
\end{lemma}

For proofs of Lemmas \ref{lemma1} and \ref{lemma2} we refer to \cite{Marynets}.

\subsection{\textit{\textbf{Main Result}}}\label{subsec3}

\begin{theorem}\label{thm1} Assume that conditions (\ref{bounded})-(\ref{spectrad}) hold for the BVP (\ref{fds})-(\ref{bcs}). Then for all fixed $\chi_{1} \in \Omega$, it holds:\\
\indent 1. Functions of the sequence (\ref{seq2}) are continuous and satisfy Dirichlet boundary conditions $u_{m}(0, \chi_{1})= \alpha_{1}$, $u_{m}(T, \chi_{1}) = \alpha_{2}$.\\
\indent 2. The sequence of functions (\ref{seq2}) for $t \in [0, T]$ converges uniformly as $ m \rightarrow \infty$ to the limit function 
\begin{align}
\label{limfun}
    u_{\infty}(t, \chi_{1}) = \lim_{m \rightarrow \infty} u_{m}(t, \chi_{1}).
\end{align}
\indent 3. The limit function satisfies boundary conditions $u_{\infty}(0, \chi_{1})=\alpha_{1}$, \\ ${u_{\infty}(T, \chi_{1})= \alpha_{2}}$.\\
\indent4. The limit function (\ref{limfun}) is a unique solution to the integral equation 
\begin{align}
\begin{split}
\label{inteq}
u(t) =& \alpha_{1} +\chi_{1}t+ (\alpha_{2}-\alpha_{1}-\chi_{1}T)\Big(\frac{t}{T}\Big)^{p} \\
    &+ \frac{1}{\Gamma(p)} \Big[ \int_{0}^{t}(t-s)^{p-1}f(s,u(s))ds -\Big(\frac{t}{T}\Big)^{p}\int_{0}^{T}(T-s)^{p-1}f(s,u(s))ds \Big],
\end{split}
\end{align}
i.e. it is a unique solution on $t \in [0, T]$ of the Cauchy problem for the modified system of FDE's:
\begin{align}
\begin{split}
\label{modsys}
     {^C_{0}}{D}^{p}_{t} u(t) &= f(t,u(t)) + \Delta(\chi_{1})\\
      u(0) &= \alpha_{1},\\
      u'(0) &= \chi_{1},
\end{split}
\end{align}
where $\Delta : \Omega \rightarrow \mathbb{R}^{n}$ is a mapping defined by
\begin{equation}
\begin{gathered}
\label{delta}
 \Delta(\chi_{1}):=  \frac{(\alpha_{2} -\alpha_{1} -\chi_{1}T)\Gamma(p+1)}{T^{p}} \\ - \frac{p}{T^{p}} \int_{0}^{T}(T-s)^{p-1}f(s, u_{\infty}(s, \chi_{1}))ds.
 \end{gathered}
\end{equation}

\indent 5. The following error estimate holds: 
\begin{align}
\label{est}
    \lvert u_{\infty}(t,\chi_{1}) - u_{m}(t,\chi_{1}) \rvert \leq &  \frac{T^{p}}{2^{2p-1}\Gamma(p+1)}  Q^{m}(I_{n} - Q)^{-1}M,
\end{align}
where $M$ and $Q$ are defined by (\ref{bounded}) and (\ref{matrix}), and $I_{n}$ is a unit $n \times n$ matrix.
\end{theorem}
\begin{proof}
The first statement follows directly from computations, since the sequence of functions (\ref{seq2}) is constructed in such a way that it satisfies the Dirichlet boundary conditions (\ref{bcs}).\\
\indent Now we prove that functions (\ref{seq2}) form a Cauchy sequence in the Banach space $C([0, T], \mathbb{R}^{n})$.  We first show that for an arbitrary point ${(t, \chi_{1}) \in [0, T] \times \Omega}$, $u_{m}(t, \chi_{1}) \in D$, $ \forall m \geq 0$. Using the estimates in (\ref{estlem1}) and (\ref{estlem2}), we find:\\
\begin{equation}
\begin{gathered}
\label{estm1}
      \lvert u_{1}(t, \chi_{1}) - u_{0}(t, \chi_{1}) \rvert\\ =  \Biggl \lvert \frac{1}{\Gamma(p)} \Big[ \int_{0}^{t}(t-s)^{p-1}f(s,u_{0}(s, \chi_{1})))ds\\ -\Big(\frac{t}{T}\Big)^{p}\int_{0}^{T}(T-s)^{p-1}f(s,u_{0}(s, \chi_{1}))ds \Big] \Biggr \rvert \\
            \leq   \alpha_{1}(t) \max_{0\leq t \leq T} \lvert f(t, u_{0}) \rvert
       \leq   \alpha_{1}(t) M
       \leq  \frac{T^{p}M}{2^{2p-1}\Gamma(p+1)} = \beta,
\end{gathered}
 \end{equation}
 where $\alpha_{1}(t)$ is given by (\ref{alpha1t}). This shows that, given an arbitrary $(t, \chi_{1}) \in [0, T] \times \Omega$, $u_{1}(t, \chi_{1}) \in D$. Similarly, by the principle of mathematical induction, for $m > 1$
\begin{equation*}
\begin{gathered}
      \lvert u_{m}(t, \chi_{1}) - u_{0}(t, \chi_{1}) \rvert \\=  \Biggl \lvert \frac{1}{\Gamma(p)} \Big[ \int_{0}^{t}(t-s)^{p-1}f(s,u_{m-1}(s,\chi_{1})))ds\\ -\Big(\frac{t}{T}\Big)^{p}\int_{0}^{T}(T-s)^{p-1}f(s,u_{m-1}(s, \chi_{1}))ds \Big] \Biggr \rvert \\
      \leq   \alpha_{1}(t) \max_{0 \leq t \leq T} \lvert f(t, u_{m-1}(s,\chi_{1})) \rvert \\
         \leq   \frac{T^{p}M}{2^{2p-1}\Gamma(p+1)} = \beta,
\end{gathered}
 \end{equation*}
 which proves that $u_{m}(t, \chi_{1}) \in D$, $\forall \,\, (t, \chi_{1}) \in [0, T] \times \Omega, \,\, m \geq 0$. \\
 \indent Now we will prove that the estimate 
 \begin{align}
 \label{est_app}
     \lvert u_{m}(t,\chi_{1}) - u_{m-1}(t,\chi_{1}) \rvert \leq K^{m-1}M \alpha_{m}(t) \leq Q^{m-1}M \alpha_{1}(t)
 \end{align}
 holds for $m \geq 1$, where $Q$ is defined in (\ref{matrix}). When $m =1$, (\ref{est_app}) follows directly from (\ref{estm1}). By induction and applying (\ref{Lipschitz}) and the estimate in (\ref{estlem2}), we obtain 
\begin{equation*}
 \begin{gathered}
   \lvert  u_{m+1}(t,\chi_{1}) - u_{m}(t,\chi_{1}) \rvert  \\
       \leq   \frac{1}{\Gamma(p)} \Bigg[ \int_{0}^{t} \Big[ (t-s)^{p-1} - \Big(\frac{t}{T} \Big)^{p}(T-s)^{p-1} \Big]\lvert f(s,u_{m}(t,\chi_{1}))-f(s,u_{m-1}(t,\chi_{1})) \rvert ds\\
      + \Big(\frac{t}{T} \Big)^{p} \int_{t}^{T}(T-s)^{p-1}\lvert f(s,u_{m}(t,\chi_{1})) - f(s,u_{m-1}(t,\chi_{1}))\rvert ds \Bigg] \end{gathered}
 \end{equation*}
 
 \begin{equation*}
 \begin{gathered}
       \leq  \frac{K}{\Gamma(p)}  \Bigg[ \int_{0}^{t} \Big[ (t-s)^{p-1} - \Big(\frac{t}{T} \Big)^{p}(T-s)^{p-1} \Big]\lvert u_{m}(t,\chi_{1})   -u_{m-1}(t,\chi_{1}) \rvert ds\\
      + \Big(\frac{t}{T} \Big)^{p} \int_{t}^{T}(T-s)^{p-1}\lvert u_{m}(t,\chi_{1}) - u_{m-1}(t,\chi_{1})\rvert ds \Bigg] \\ 
      \leq  K^{m}M \frac{1}{\Gamma(p)}  \Biggl[ \int_{0}^{t} \Big[ (t-s)^{p-1} - \Big(\frac{t}{T} \Big)^{p}(T-s)^{p-1} \Big]\alpha_{m}(s)ds\\
      + \Big(\frac{t}{T} \Big)^{p} \int_{t}^{T}(T-s)^{p-1}\alpha_{m}(s)ds \Biggr] \\
       = K^{m}M \alpha_{m+1}(t) \leq  K^{m}M \frac{T^{mp}\alpha_{1}(t)}{2^{[m(2p-1)]}[\Gamma(p+1)]^{m}} = Q^{m}M\alpha_{1}(t),
\end{gathered}
 \end{equation*}
for all $ t \in [0, T],\, u_{0} \in D$. In view of (\ref{est_app}), we get the estimate 
 \begin{equation*}
 \begin{gathered}
     \lvert u_{m+j}(t,\chi_{1}) - u_{m}(t,\chi_{1}) \rvert 
     = \sum_{k=1}^{j} \lvert u_{m+k}(t,\chi_{1}) - u_{m+k-1}(t,\chi_{1}) \rvert \\
     \leq  \sum_{k=1}^{j} K^{m+k-1}M\alpha_{m+k}(t)
     \leq  \sum_{k=1}^{j} \frac{K^{m+k-1}(T-t_{1})^{p(m+k-1)}M\alpha_{1}(t)}{2^{(m+k-1)(2p-1)}[\Gamma(p+1)]^{m+k-1}}\\
     = \sum_{k=0}^{j-1} Q^{m+k}M \alpha_{1}(t) = Q^{m} \sum_{k=0}^{j-1} Q^{k}M\alpha_{1}(t).
\end{gathered}
 \end{equation*}
 Since $r(Q) < 1$, it holds that
 \begin{align*}
     \lim_{n \rightarrow \infty } \sum_{k = 0}^{n} Q^{k} \leq (I_{n}- Q)^{-1} \,\,\,\,\, \text{and} \,\,\,\,\, \lim_{m \rightarrow \infty} Q^{m} = O_{n},
 \end{align*}
 where $O_{n}$ denotes the $n \times n$ matrix of zeros. Passing in the last inequality to the limit when $ j \rightarrow \infty$, we obtain the estimate in (\ref{est}). Thus, the sequence of functions in (\ref{seq2}) converges uniformly to the limit function $u_{\infty}(t, \chi_{1})$ in the domain $[0, T] \times D$, according to the Cauchy criteria.\\ 
 \indent Since $u_{\infty}(t,\chi_{1})$ is the limit of a sequence of functions (\ref{seq2}), all of which satisfy boundary conditions (\ref{bcs}), $u_{\infty}(t,\chi_{1})$ also satisfies the boundary conditions. Passing in (\ref{seq2}) to the limit $ m \rightarrow \infty$, we get that the function $u_{\infty}(t, \chi_{1})$ is a solution to the integral equation (\ref{inteq}).\\
 \indent Next, we show that the integral equation (\ref{inteq}) has a unique continuous solution. Suppose $u_{1}(t)$ and $u_{2}(t)$ are two distinct solutions to (\ref{inteq}). Then 
 \begin{equation*}
 \begin{gathered}
     \lvert u_{1}(t) - u_{2}(t) \rvert \leq  \frac{K}{\Gamma(p)}  \Bigg[\int_{0}^{t}(t-s)^{p-1} \lvert u_{1}(s)-u_{2}(s)\rvert ds\\
     +\Big(\frac{t}{T}\Big)^{p}\int_{0}^{T}(T-s)^{p-1}\lvert u_{1}(s)-u_{2}(s) \rvert ds\Bigg] \\
      = K \alpha_{1}(t)\max_{0\leq s \leq T}\lvert u_{1}(s) - u_{2}(s) \rvert
      \leq  \frac{KT^{p}}{2^{2p-1}\Gamma(p+1)}\max_{0\leq s \leq T} \lvert u_{1}(s) - u_{2}(s)\rvert \\
      =  Q \max_{0\leq s \leq T} \lvert u_{1}(s) - u_{2}(s)\rvert,\\
 \end{gathered}
  \end{equation*}
for all $ s \in [0, T]$. Thus, the inequality 
 \begin{align*}
   \max_{0\leq t \leq T}  \lvert u_{1}(t) - u_{2}(t) \rvert \leq Q \max_{0\leq t \leq T} \lvert u_{1}(t) - u_{2}(t) \rvert
 \end{align*}
 holds, which implies $\underset{0\leq t \leq T}{\max}\lvert u_{1}(t) - u_{2}(t) \rvert  = 0 $, since $r(Q) < 1$. Thus, $u_{1}(t) = u_{2}(t)$ for all $t \in [0, T]$. Moreover, the initial value problem (IVP) (\ref{modsys}) is equivalent to the integral equation \cite{survey}
 \begin{align}
 \begin{split}
 \label{inteq2}
       u(t) &= \alpha_{1} +\chi_{1}t+ \frac{1}{\Gamma(p)}\int_{t_{1}}^{t}(t-s)^{p-1}[f(s, u(s)) + \Delta(\chi_{1})]ds\\
    &=  \alpha_{1} + \chi_{1}t+\frac{1}{\Gamma(p)}\int_{t_{1}}^{t}(t-s)^{p-1}f(s, u(s)) ds + \frac{(t-t_{1})^{p}\Delta(\chi_{1})}{\Gamma(p+1)} \\
    &+ \frac{1}{\Gamma(p)} \Big[ \int_{0}^{t}(t-s)^{p-1}f(s,u(s))ds -\Big(\frac{t}{T}\Big)^{p}\int_{0}^{T}(T-s)^{p-1}f(s,u_{\infty}(s, \chi_{1}))ds \Big],\\
\end{split}
 \end{align}
 where the perturbation $\Delta(\chi_{1})$ is given by (\ref{delta}). Comparing (\ref{inteq}) and (\ref{inteq2}) and recalling that $u_{\infty}(t, \chi_{1})$ is the unique continuous solution of ($\ref{inteq}$), it follows that $u(t) = u_{\infty}(t, \chi_{1})$ in (\ref{inteq2}), i.e. $u_{\infty}(t, \chi_{1})$ is the unique continuous solution of (\ref{modsys}). This completes the proof.
\end{proof}

Next, we show the connection between the solution to the IVP (\ref{modsys}) and the original BVP.

\section{Connection of the limit function to BVP}\label{sec4}
Consider the Cauchy problem 
\begin{subequations}
\label{eq:Cauchy}
\begin{align}
  {^C_{0}}{D}^{p}_{t} u(t) &= f(t,u(t)) + \mu, \,\,\, t \in [0, T], \label{eq:Cauchy1}\\
      u(0) &= \alpha_{1}, \label{eq:Cauchy2}\\
      u'(0) &= \chi_{1} \label{eq:Cauchy3},
\end{align}
\end{subequations}
where $\mu \in \mathbb{R}^{n}$ we will call a control parameter, $\alpha_{1} \in D_{\beta}$ and $\chi_{1} \in \Omega$. \\

\begin{theorem}\label{thm2} Let $\chi_{1} \in \Omega$, $\mu \in \mathbb{R}^{n}$ be given vectors. Assume that all conditions of Theorem~\ref{thm1} are satisfied for the FDS (\ref{fds}). Then the solution $u = u( \cdot, \chi_{1}, \mu)$ of the IVP (\ref{eq:Cauchy}) also satisfies boundary conditions (\ref{bcs}) if and only if 
\begin{align}
\label{thm21}
    \mu = \Delta(\chi_{1}),
\end{align}
where $\Delta(\chi_{1})$ is given by (\ref{delta}), and in this case
\begin{align}
\label{thm22}
    u(t,\chi_{1},\mu) = u_{\infty}(t, \chi_{1})\,\,\, \text{\emph{for}}\,\,\, t \in [0, T].
\end{align}
\end{theorem}
\begin{proof}
First note that the existence and uniqueness of the solution to the IVP (\ref{eq:Cauchy}) on $t \in [0, T]$ and its continuous dependence on $\chi_{1}$ and $\mu$ follow from the theory in \cite{kilbas}.

\medskip
\indent \emph{Sufficiency.} Suppose that
\begin{align*}
    \mu = \Delta(\chi_{1}).
\end{align*}
By Theorem~\ref{thm1}, it follows that the limit function $u_{\infty}(t, \chi_{1})$ of the sequence (\ref{seq2}) is a unique solution to equation (\ref{eq:Cauchy1}), which satisfies boundary conditions (\ref{bcs}). Moreover, the limit function $u_{\infty}(t, \chi_{1})$ also satisfies the initial conditions (\ref{eq:Cauchy2}), (\ref{eq:Cauchy3}). Thus, it is the unique solution to the Cauchy problem (\ref{eq:Cauchy}) for $ \mu = \Delta(\chi_{1})$, and $ u(t,\chi_{1},\mu) = u_{\infty}(t, \chi_{1})$ holds. This also means that the equality in (\ref{thm22}) takes place.

\medskip
\indent \emph{Necessity.} Now we show that the parameter value in (\ref{thm21}) is unique. Suppose that there exists another parameter $\bar{\mu}$, such that the solution $\bar{u}(t, \chi_{1})$ to the IVP 
\begin{align*}
     {^C_{0}}{D}^{p}_{t} u(t) &= f(t, u(t)) + \bar{\mu}, \,\,\, t \in [0, T], \\
      u(0) &=  \alpha_{1}, \\
      u'(0) &= \chi_{1}, \,
\end{align*}
also satisfies the boundary conditions in (\ref{bcs}). Then, according to (\cite{kilbas}, Cor. 3.24), the function $\bar{u}(t, \chi_{1})$ is also a continuous solution to the integral equation
\begin{align}
\label{inteq1}
    \bar{u}(t, \chi_{1}) =& \alpha_{1} + \chi_{1}t + \frac{1}{\Gamma(p)}\int_{0}^{t} (t-s)^{p-1}f(s, \bar{u}(s, \chi_{1}) ds + \frac{t^{p}\bar{\mu}}{\Gamma(p+1)}.
\end{align}
Moreover, $\bar{u}(t, \chi_{1})$ satisfies the boundary conditions in (\ref{bcs}) and the initial condition (\ref{eq:Cauchy3}), that is,
\begin{align*}
    \bar{u}(0, \chi_{1}) =& \alpha_{1},\\
     \bar{u}(T, \chi_{1}) =& \alpha_{2},\\
     \bar{u}{'}(0) =& \chi_{1}.
\end{align*}
Substituting this into equation (\ref{inteq1}) for $t= T$, we obtain 
\begin{align}
\label{mubar}
    \bar{\mu} &=  \frac{(\alpha_{2} -\alpha_{1} -\chi_{1}T)\Gamma(p+1)}{T^{p}} - \frac{p}{(T-t_{1})^{p}} \int_{0}^{T}(T-s)^{p-1}f(s, u(s))ds.
\end{align}
Plugging (\ref{mubar}) into (\ref{inteq1}) yields
\begin{align}
\begin{split}
\label{inteq3}
    \bar{u}(t, \chi_{1}) =& \alpha_{1} +\chi_{1}t+ (\alpha_{2}-\alpha_{1}-\chi_{1}T)\Big(\frac{t}{T}\Big)^{p} \\
    &+ \frac{1}{\Gamma(p)} \Big[ \int_{0}^{t}(t-s)^{p-1}f(s,u(s))ds -\Big(\frac{t}{T}\Big)^{p}\int_{0}^{T}(T-s)^{p-1}f(s,u(s))ds \Big].
\end{split}
\end{align}
Since $\alpha_{1} \in D_{\beta}$, according to the integral equation (\ref{inteq3}) and the definition of $D_{\beta}$, it can be shown that $\bar{u}(t, \chi_{1}) \in D$. Moreover, since Equations (\ref{inteq}) and (\ref{inteq3}) are equivalent, it follows from part 4 of Theorem~\ref{thm1} that $\bar{u}(t, \chi_{1}) = u_{\infty}(t, \chi_{1})$ and $\mu = \Delta(\chi_{1})$.  This completes the proof. 
\end{proof}

\begin{theorem}\label{thm3}
 Let the original BVP (\ref{fds}), (\ref{bcs}) satisfy conditions (\ref{bounded})-(\ref{spectrad}). Then $u_{\infty}( \cdot, \chi_{1}^{*})$ is a solution to the FDS (\ref{fds}) with boundary conditions  (\ref{bcs}) if and only if the point $\chi_{1}^{*}$ is a solution to the determining equation 
\begin{align}
\label{deteq}
    \Delta(\chi_{1}^{*}) = 0,
\end{align}
where $\Delta$ is given by (\ref{delta}). \\
\end{theorem}
\begin{proof}
The conditions of Theorem~\ref{thm1} hold, thus we can apply Theorem \ref{thm2} and note that the perturbed equation in (\ref{modsys}) coincides with the original FDS (\ref{fds}) if and only if the point $\chi_{1}^{*}$ satisfies the determining equation (\ref{deteq}). That is, $u_{\infty}( \cdot, \chi_{1}^{*})$ is a solution to the BVP (\ref{fds}), (\ref{bcs}) if and only if (\ref{deteq}) holds. \\
\end{proof}

In the following section we give sufficient and necessary conditions for the existence of solutions to the BVP (\ref{fds}), (\ref{bcs}).

\section{Solvability Analysis}

\begin{lemma}\label{lem3}Suppose the conditions of Theorem~\ref{thm1} are satisfied. Then for arbitrary $m \geq 1$ and $\chi_{1} \in \Omega$ for the exact and approximate determining functions $\Delta: \Omega \rightarrow
 \mathbb{R}^{n}$ and $\Delta_{m}: \Omega \rightarrow
 \mathbb{R}^{n}$, defined by (\ref{delta}) and 
 \begin{equation}
  \begin{gathered}
\label{deteqapp}
     \Delta_{m}(\chi_{1}) :=  \frac{(\alpha_{2} -\alpha_{1} -\chi_{1}T)\Gamma(p+1)}{T^{p}} \\- \frac{p}{T^{p}} \int_{0}^{T}(T-s)^{p-1}f(s, u_{m}(s, \chi_{1}))ds,
\end{gathered}
 \end{equation}
respectively, the inequality
\begin{align}
\label{dineq}
\lvert \Delta(\chi_{1}) - \Delta_{m}(\chi_{1}) \rvert \leq Q^{m}M(I_{n}-Q)^{-1}
\end{align}
holds, where $M,\,\,K$ and $Q$ are given in (\ref{bounded}), (\ref{Lipschitz}), and (\ref{matrix}).\\
\end{lemma}
\begin{proof}
Let us fix an arbitrary $\chi_{1} \in \Omega$. Then by virtue of the Lipschitz condition (\ref{Lipschitz}) and the estimates in (\ref{est}) and (\ref{estlem2}), we have

 \begin{equation*}
  \begin{gathered}
    \lvert \Delta(\chi_{1}) - \Delta_{m}(\chi_{1}) \rvert \\= \Biggr \rvert - \frac{p}{T^{p}} \int_{0}^{T} (T - s)^{p-1}f(s, u_{\infty}(s, \chi_{1})) ds
    + \frac{p}{T^{p}} \int_{0}^{T} (T - s)^{p-1}f(s, u_{m}(s, \chi_{1})) ds  \Biggl \rvert \\
    \leq \frac{p}{T^{p}} \int_{0}^{T}(T-s)^{p-1}\lvert f(s, u_{\infty}(s, \chi_{1}))-f(s, u_{m}(s, \chi_{1}))\rvert ds\\
    \leq \frac{p K}{T^{p}} \int_{0}^{T}(T-s)^{p-1}\lvert u_{\infty}(s, \chi_{1})-u_{m}(s, \chi_{1})\rvert ds\\
    \leq Q^{m}M(I_{n}-Q)^{-1}.
 \end{gathered}
 \end{equation*}
The obtained estimate proves the lemma.
\end{proof}

On the basis of the exact and approximate determining equations (\ref{deteq}) and 
\begin{align}
\label{exdeteq}
    \Delta_{m}(\chi_{1}) = 0,
\end{align}
let us introduce the mappings $\Phi: \mathbb{R}^{n} \rightarrow \mathbb{R}^{n}$ and $\Phi_{m}: \mathbb{R}^{n} \rightarrow \mathbb{R}^{n}$, defined by
\begin{subequations}
\begin{align}
    \Phi(\chi_{1}) & := \frac{(\alpha_{2} -\alpha_{1} -\chi_{1}T)\Gamma(p+1)}{T^{p}} - \frac{p}{T^{p}} \int_{0}^{T}(T-s)^{p-1}f(s, u_{\infty}(s, \chi_{1}))ds, \label{phi}\\
    \Phi_{m}(\chi_{1}) & := \frac{(\alpha_{2} -\alpha_{1} -\chi_{1}T)\Gamma(p+1)}{T^{p}} - \frac{p}{T^{p}} \int_{0}^{T}(T-s)^{p-1}f(s, u_{m}(s, \chi_{1}))ds, \label{phim}
\end{align}
\end{subequations}
and recall the following definition presented in \cite{RontoMarynets}:\\

\begin{definition} Let $H \subset \mathbb{R}^{n}$ be a non-empty set. For any pair of functions 
\begin{align*}
    f_{j} = col(f_{j,1}(x),...,f_{j,n}(x)): H \rightarrow \mathbb{R}^{n},\,\,\, j = 1,2
\end{align*}
the following statement holds 
\begin{align*}
    f_{1} \triangleright_{\tiny{H}} f_{2} 
\end{align*}
if and only if there exists a function $k : H \rightarrow \{1, 2, .... ,n \}$, such that 
\begin{align*}
    f_{1,k(x)} > f_{2,k(x)}
\end{align*}
for all $x \in H$. It means that at least one of the components of $f_{1}(x)$ is less than the appropriate component of $f_{2}(x)$ in every point in $H$.\\
\end{definition}
\begin{theorem}\label{thm4} Suppose the conditions of Theorem~\ref{thm1} hold, and one can find an $m \geq 1$ and a set $\Omega$, such that
\begin{align}
\label{map}
    \Phi_{m} \triangleright_{\partial \Omega }  Q^{m}M(I_{n}-Q)^{-1}.
\end{align}
If the Brouwer degree of the mapping $\Phi_{m}$ satisfies  \\
\begin{align}
\label{brouwer}
    \text{deg}(\Phi_{m}, \Omega, 0) \neq 0,
\end{align}
then there exists a point $\chi_{1}^{*} \in \Omega$, such that 
\begin{align}
\label{sol}
    u_{\infty}(t) = u_{\infty}(t, \chi_{1}^{*}) = \lim_{m \rightarrow \infty} u_{m}(t, \chi_{1}^{*})
\end{align}
is a solution to the BVP (\ref{fds}), (\ref{bcs}) satisfying 
\begin{align}
\label{chistar}
    u{'}_{\infty}(0) = \chi_{1}^{*} \in \Omega.
\end{align}
\end{theorem}
\begin{proof}
We first show that the vector fields $\Phi$ and $\Phi_{m}$ are homotopic. Let us introduce the family of vector mappings
\begin{align}
\label{vecmaps}
    P(\theta, \chi_{1}) = \Phi_{m}(\chi_{1}) + \theta [\Phi(\chi_{1}) - \Phi_{m}(\chi_{1}) ],\,\,\, \chi_{1} \in \partial \Omega,\,\,\, \theta \in [0,1].
\end{align}
Then $P(\theta, \chi_{1})$ is continuous for all $\chi_{1} \in \partial \Omega,\,\, \theta \in [0,1]$. We have  
\begin{align*}
    P(0, \chi_{1}) = \Phi_{m}(\chi_{1}),\,\,\, P(1, \chi_{1})  = \Phi(\chi_{1})
\end{align*}
and for any $\chi_{1} \in \Omega$, 
\begin{align}
\begin{split}
\label{Pmap}
    \lvert P(\theta, \chi_{1}) \rvert =& \lvert \Phi_{m}(\chi_{1}) + \theta [\Phi(\chi_{1}) - \Phi_{m}(\chi_{1}) ]\rvert \\
    \geq& \lvert \Phi_{m}(\chi_{1}) \rvert - \lvert \Phi(\chi_{1}) - \Phi_{m}(\chi_{1}) \rvert.
\end{split}
\end{align}
From the other side, by virtue of (\ref{phi}), (\ref{phim}) we have 
\begin{align}
\label{mapbound}
\lvert \Phi(\chi_{1}) - \Phi_{m}(\chi_{1}) \rvert \leq Q^{m}M(I_{n}-Q)^{-1}.
\end{align}
From (\ref{map}), (\ref{Pmap}), and (\ref{mapbound}) it follows that
\begin{align*}
    \lvert P(\theta, \chi_{1}) \rvert \triangleright_{\partial \Omega} 0,\,\,\, \theta \in [0,1],
\end{align*}
which means that $P(\theta, \chi_{1}) \neq 0$ for all $\theta \in [0,1]$ and $\chi_{1} \in \Omega$, i.e. the mappings (\ref{vecmaps}) are non-degenerate, and thus the vector fields $\Phi$ and $\Phi_{m}$ are homotopic. Since relation (\ref{brouwer}) holds and the Brouwer degree is preserved under homotopies, it follows that
\begin{align*}
      \text{deg}(\Phi, \Omega, 0) = \text{deg}(\Phi_{m}, \Omega, 0) \neq 0.
\end{align*}
which implies that there exists $\chi_{1}^{*} \in \Omega$ such that $\Phi(\chi_{1}^{*}) = 0$ by the classical topological result in \cite{farkas}. \\
\indent Hence, the point $\chi_{1}^{*}$ satisfies the determining equation (\ref{deteq}). \\
\indent By Theorem \ref{thm3} it follows that the function defined in (\ref{sol}) is a solution to the original BVP with the Dirichlet boundary conditions (\ref{fds}), (\ref{bcs}) and satisfies the initial condition (\ref{chistar}).
\end{proof}

\begin{lemma}\label{lem4}Suppose the conditions of Theorem~\ref{thm1} are satisfied. Then the limit function $u_{\infty}(t, \chi_{1})$ satisfies the Lipschitz-type condition of the form
\begin{align}
\label{limlip}
\lvert u_{\infty}(t, \chi_{1}^{0}) - u_{\infty}(t, \chi_{1}^{1}) \rvert \leq \Big[ R + \alpha_{1}(t)R(I_{n}-Q)^{-1}\Big] \lvert \chi_{1}^{0}-\chi_{1}^{1} \rvert,
\end{align}
where
\begin{align}
\label{R}
    R := \sup_{t \in [0,T]} \Biggl \lvert t - T\Big(\frac{t}{T}\Big)^{p}\Biggr \rvert.
\end{align}
\end{lemma}
\begin{proof}
 Using (\ref{seq2}) for $m =1$, we find that
\begin{equation*}
\begin{gathered}
\lvert  u_{1}(t, \chi_{1}^{0}) - u_{1}(t, \chi_{1}^{1}) \rvert \leq \lvert \chi_{1}^{0}-\chi_{1}^{1}\rvert R\\
+ \frac{1}{\Gamma(p)}  \int_{0}^{t}\Big[(t-s)^{p-1}-(T-s)^{p-1}\Big(\frac{t}{T}\Big)^{p}\Big]\lvert f(s, u_{0}(s, \chi_{1}^{0}))-f(s, u_{0}(s, \chi_{1}^{1}))\rvert ds \\ 
+\frac{1}{\Gamma(p)} \Big(\frac{t}{T}\Big)^{p} \int_{t}^{T}(T-s)^{p-1}\lvert f(s, u_{0}(s, \chi_{1}^{0}))-f(s, u_{0}(s, \chi_{1}^{1}))\rvert ds  \\
\leq  \lvert \chi_{1}^{0}-\chi_{1}^{1}\rvert R+ \frac{K}{\Gamma(p)}\int_{0}^{t}\Big[(t-s)^{p-1}-(T-s)^{p-1}\Big(\frac{t}{T}\Big)^{p}\Big]\lvert u_{0}(s, \chi_{1}^{0})-u(s, \chi_{1}^{1})\rvert ds \\ 
+\frac{K}{\Gamma(p)} \Big(\frac{t}{T}\Big)^{p} \int_{t}^{T}(T-s)^{p-1}\lvert u_{0}(s, \chi_{1}^{0})-u(s, \chi_{1}^{1})\rvert ds  \\
\leq \lvert \chi_{1}^{0}-\chi_{1}^{1}\rvert R + \frac{KR}{\Gamma(p)}\lvert \chi_{1}^{0}-\chi_{1}^{1}\rvert \int_{0}^{t}\Big[(t-s)^{p-1}-(T-s)^{p-1}\Big(\frac{t}{T}\Big)^{p}\Big]ds \\ 
+\frac{KR}{\Gamma(p)}\lvert \chi_{1}^{0}-\chi_{1}^{1}\rvert \Big(\frac{t}{T}\Big)^{p} \int_{t}^{T}(T-s)^{p-1}ds   = \lvert \chi_{1}^{0}-\chi_{1}^{1}\rvert R + KR \alpha_{1}(t)\lvert \chi_{1}^{0}-\chi_{1}^{1}\rvert 
\end{gathered}
\end{equation*}
holds for all $t \in [0, T]$, where the matrix $K$ and vector $R$ are defined in (\ref{Lipschitz}) and (\ref{R}), and the function $\alpha_{1}(t)$ is defined in (\ref{alpha1t}). Analogously, for $m=2$ we find 
\begin{equation*}
\begin{gathered}
     \lvert  u_{2}(t, \chi_{1}^{0})-u_{2}(t, \chi_{1}^{1})\rvert \\
     \leq
    \lvert \chi_{1}^{0}-\chi_{1}^{1} \rvert R + \frac{K}{\Gamma(p)}\int_{0}^{t}\Big[(t-s)^{p-1}-(T-s)^{p-1}\Big(\frac{t}{T}\Big)^{p}\Big]\lvert u_{1}(t, \chi_{1}^{0})-u_{1}(t, \chi_{1}^{1})\rvert ds \\
    +\frac{K}{\Gamma(p)} \Big(\frac{t}{T}\Big)^{p} \int_{t}^{T}(T-s)^{p-1}\lvert u_{1}(t, \chi_{1}^{0})-u_{1}(t, \chi_{1}^{1})\rvert ds\\
    = [R + KR\alpha_{1}(t) +K^{2}\alpha_{2}(t)]\lvert \chi_{1}^{0}-\chi_{1}^{1}\rvert.
\end{gathered}
\end{equation*}
By induction we get:
\begin{equation*}
\begin{gathered}
     \lvert u_{m}(t, \chi_{1}^{0})-u_{m}(t, \chi_{1}^{1}) \rvert \\
     \leq \Big[R + \sum_{i=1}^{m-1}K^{i}R\alpha_{i}(t) + K^{m}\alpha_{m}(t) \Big]\lvert \chi_{1}^{0}-\chi_{1}^{1}\rvert \\
      \leq \Big[R + \sum_{i=1}^{m-1}Q^{i}R\alpha_{1}(t) + Q^{m} \Big]\lvert \chi_{1}^{0}-\chi_{1}^{1}\rvert \\
    \leq \Big[ R + R \alpha_{1}(t)(I_{n}-Q)^{-1} + Q^{m} \Big] \lvert \chi_{1}^{0}-\chi_{1}^{1}\rvert ,
\end{gathered}
\end{equation*}
and passing to the limit $m \rightarrow \infty$ in the inequality above yields 
\begin{align*}
    \lvert u_{\infty}(t, \chi_{1}^{0}) - u_{\infty}(t, \chi_{1}^{1})\rvert \leq \Big[ R + \alpha_{1}(t)R(I_{n}-Q)^{-1}\Big] \lvert \chi_{1}^{0}-\chi_{1}^{1} \rvert,
\end{align*}
as required.
\end{proof}

\begin{lemma}\label{lem5}Suppose the conditions of Theorem~\ref{thm1} are satisfied. Then the function $\Delta: \Omega \rightarrow \mathbb{R}^{n}$ satisfies the following estimate:
\begin{align}
\label{lemma7}
    \lvert \Delta(\chi_{1}^{0}) - \Delta(\chi_{1}^{1})\rvert \leq& \frac{\Gamma(p+1)}{T^{p-1}}\lvert\chi_{1}^{0}-\chi_{1}^{1}\rvert + \Big( KR +QR(I_{n}-Q)^{-1}\Big)\lvert\chi_{1}^{0}-\chi_{1}^{1}\rvert.
\end{align}
\end{lemma}
\begin{proof}
From (\ref{delta}) we have 
\begin{equation*}
\begin{gathered}
    \Delta(\chi_{1}^{0}) - \Delta(\chi_{1}^{1}) = \frac{\Gamma(p+1)}{T^{p-1}}(\chi_{1}^{1}-\chi_{1}^{0}) \\
    + \frac{p}{T^{p}} \int_{0}^{T}(T-s)^{p-1}[f(s, u_{\infty}(s, \chi_{1}^{1})) - f(s, u_{\infty}(s, \chi_{1}^{0}))]ds.
  \end{gathered}
\end{equation*}

Applying (\ref{Lipschitz}) and (\ref{limlip}) yields
\begin{equation*}
\begin{gathered}
     \lvert  \Delta(\chi_{1}^{0}) - \Delta(\chi_{1}^{1})\rvert \leq \frac{\Gamma(p+1)}{T^{p-1}}\lvert\chi_{1}^{0}-\chi_{1}^{1}\rvert\\
     + \frac{pK}{T^{p}} \int_{0}^{T}(T-s)^{p-1}\lvert u_{\infty}(s, \chi_{1}^{0}) - u_{\infty}(s, \chi_{1}^{1})\rvert ds\\
     \leq \frac{\Gamma(p+1)}{T^{p-1}}\lvert\chi_{1}^{0}-\chi_{1}^{1}\rvert + \Big(KR + QR(I_{n}-Q)^{-1} \Big)\lvert \chi_{1}^{0}-\chi_{1}^{1}\rvert,
  \end{gathered}
\end{equation*}
as required.
\end{proof}

\begin{theorem}\label{thm5}Suppose the conditions of Theorem~\ref{thm1} are satisfied. Then in order for the domain $\Omega$ to contain a point $\chi_{1} = \chi_{1}^{*}$, which determines the value of the first derivative, $u'(0, \chi_{1}^{*})$, of the solution $u(t, \chi_{1})$ of the BVP (\ref{fds}), (\ref{bcs}) at $t = 0$, it is necessary that for all $m \geq 1,$ $\tilde{\chi_{1}} \in \Omega$, the following inequality holds:
\begin{equation*}
\begin{gathered}
    \lvert \Delta_{m}(\tilde{\chi_{1}})\rvert \leq \sup_{\chi_{1} \in \Omega} \Big[KR + \frac{QR}{1-Q} + \frac{\Gamma(p+1)}{T^{p-1}} \Big]\lvert \chi_{1} - \tilde{\chi_{1}}\rvert + \frac{Q^{m}M}{1-Q}.
  \end{gathered}
\end{equation*}
\end{theorem}
\begin{proof}
Assume that the determining function $\Delta(\chi_{1})$ vanishes at $\chi_{1} = \chi_{1}^{*}$, i.e. ${\Delta(\chi_{1}^{*}) = 0}$. Then, according to Theorem \ref{thm3}, the initial value of the first derivative of the solution of BVP (\ref{fds}), (\ref{bcs}), is given by $u'(0) = \chi_{1}^{*}$.\\

Let us apply Lemma \ref{lem5}, where $\chi_{1}^{0} = \tilde{\chi_{1}}$ and $\chi_{1}^{1} = \chi_{1}^{*}$:
\begin{equation*}
    \lvert \Delta(\tilde{\chi_{1}})-\Delta(\chi_{1}^{*})\rvert  = \lvert \Delta(\tilde{\chi_{1}}) \rvert 
    \leq \Big[KR + \frac{QR}{1-Q} + \frac{\Gamma(p+1)}{T^{p-1}} \Big]\lvert \tilde{\chi_{1}} - \chi_{1}^{*}\rvert.
\end{equation*}

By Lemma \ref{lem3}, it follows that
\begin{equation*}
    \lvert \Delta(\tilde{\chi_{1}}) - \Delta_{m}(\tilde{\chi_{1}} \rvert \leq \frac{Q^{m}M}{1-Q},
\end{equation*}
thus,
\begin{equation*}
\begin{gathered}
    \lvert \Delta_{m}(\tilde{\chi_{1}}) \rvert \leq \lvert \Delta(\tilde{\chi_{1}})\rvert  + \lvert \Delta_{m}(\tilde{\chi_{1}}) - \Delta(\tilde{\chi_{1}})\rvert \\
    \leq \Big[KR + \frac{QR}{1-Q} + \frac{\Gamma(p+1)}{T^{p-1}} \Big]\lvert \tilde{\chi_{1}} - \chi_{1}^{*}\rvert + \frac{Q^{m}M}{1-Q}\\
    \leq \sup_{\chi_{1} \in \Omega}  \Big[KR + \frac{QR}{1-Q} + \frac{\Gamma(p+1)}{T^{p-1}} \Big]\lvert \tilde{\chi_{1}} - \chi_{1}\rvert + \frac{Q^{m}M}{1-Q}.\\
  \end{gathered}
\end{equation*}
This proves the theorem.
\end{proof}

\begin{remark}
On the basis of Theorem \ref{thm5}, we can establish an algorithm of approximate search for the point $\chi_{1}^{*}$, which defines the solution $u(\cdot)$ of the original BVP (\ref{fds}), (\ref{bcs}). Let us represent the open set $\Omega \subset \mathbb{R}^{n}$ as the finite union of subsets $\Omega_{i}$:
\begin{align}
\label{union}
    \Omega = \cup_{i=1}^{N} \Omega_{i}.
\end{align}
In each subset $\Omega_{i}$, we pick a point $\tilde{\chi_{1}}^{i}$ and calculate the approximate solution $u_{m}(t, \tilde{\chi_{1}}^{i})$ using the recurrence formula (\ref{seq2}). Then we find the value of the determining function $\Delta_{m}(\tilde{\chi_{1}}^{i})$, according to (\ref{deteq}), and exclude from (\ref{union}) subsets $\Omega_{i}$ for which the inequality does not hold. According to Theorem \ref{thm5}, these subsets cannot contain a point $\chi_{1}^{*}$ that determines the solution $u(\cdot)$. The remaining subsets $\Omega_{i_{1}}, ..., \Omega_{i_{s}}$ form a set $\Omega_{m,N}$, such that only $\tilde{\chi_{1}} \in \Omega_{m,N}$ can determine $u(\cdot)$. \\
As $N, m \rightarrow \infty$, the set $\Omega_{m,N}$ "follows" the set $\Omega^{*}$, which may contain a value $\chi_{1}^{*}$ and defines a solution to the BVP (\ref{fds}), (\ref{bcs}). Each point $\tilde{\chi_{1}}$ can be seen as an approximation of $\chi_{1}^{*}$, which determines solution of the BVP (\ref{fds}), (\ref{bcs}). It is clear that 
\begin{align*}
    \lvert \tilde{\chi_{1}} - \chi_{1}^{*}\rvert \leq \sup_{\chi_{1} \in \Omega_{m,N}} \lvert \tilde{\chi_{1}} - \chi_{1}\rvert ,
\end{align*}
and the function $u_{m}(t, \tilde{\chi_{1}})$, calculated using the iterative formula (\ref{seq2}), can be seen as an approximate solution to the BVP (\ref{fds}), (\ref{bcs}).\\
\end{remark} 

\begin{theorem}\label{thm6} Suppose the conditions of Theorem 1 are satisfied and a point $\chi_{1}^{*}$, defined in the set $\Omega$, is the solution of the exact determining equation (\ref{deteq}), and $\tilde{\chi_{1}}$ is an arbitrary point in the set $\Omega_{m,N}$. Then the following estimate holds: 
\begin{align*}
    \lvert u_{\infty}(t, \chi_{1}^{*}) - u_{m}(t, \tilde{\chi_{1}})\rvert \leq& Q^{m}M(I_{n}-Q)^{-1}\alpha_{1}(t)\\
    &+ \sup_{\tilde{\chi_{1}}\in \Omega_{m,N}} \Big(R + R\alpha_{1}(t)(I_{n}-Q)^{-1} + Q^{m} \Big)\lvert \chi_{1}^{*} - \tilde{\chi_{1}}\rvert.
\end{align*}
\end{theorem}
\begin{proof}
Let us use the following inequality:
\begin{align*}
      \lvert u_{\infty}(t, \chi_{1}^{*}) - u_{m}(t, \tilde{\chi_{1}})\rvert \leq& \lvert u_{\infty}(t, \chi_{1}^{*}) - u_{m}(t, \chi_{1}^{*})\rvert + \lvert u_{m}(t, \chi_{1}^{*}) -  u_{m}(t, \tilde{\chi_{1}})\rvert.
\end{align*}
According to the estimate in (\ref{est}), we have 
\begin{align*}
    \lvert u_{\infty}(t, \chi_{1}^{*}) - u_{m}(t, \chi_{1}^{*})\rvert  \leq& Q^{m}(I_{n}-Q)^{-1}M\alpha_{1}(t).
\end{align*}
Moreover, from the estimate in Lemma \ref{lem4}, it follows that
\begin{align*}
    \lvert u_{m}(t, \chi_{1}^{*}) -  u_{m}(t, \tilde{\chi_{1}})\rvert \leq& \Big( R + R\alpha_{1}(t)(I_{n}-Q)^{-1} + Q^{m}\Big) \lvert \chi_{1}^{*} - \tilde{\chi_{1}}\rvert.
\end{align*}
Therefore, we find
\begin{equation*}
\begin{gathered}
    \lvert u_{\infty}(t, \chi_{0}^{*}) - u_{m}(t, \tilde{\chi_{1}})\rvert \\
    \leq \frac{Q^{m}}{I_n-Q}M\alpha_{1}(t) + \Big( R + \frac{R\alpha_{1}(t)}{I_n-Q} + Q^{m}\Big) \lvert \chi_{1}^{*} - \tilde{\chi_{1}}\rvert \\
    \leq \frac{Q^{m}}{I_n-Q}M\alpha_{1}(t) + \sup_{\tilde{\chi_{1}}\in \Omega_{m,N}} \Big(R + \frac{R\alpha_{1}(t)}{I_n-Q} + Q^{m} \Big)\lvert \chi_{1}^{*} - \tilde{\chi_{1}}\rvert,
    \end{gathered}
\end{equation*}
as required.
\end{proof}

In the following section we apply the numerical-analytic technique to a particular model example.

\section{Example}\label{sec5}

Motivated by \cite{marynets1}, we consider the BVP for the fractional differential equation
\begin{align}
\label{egeq}
 {^C_{0}}{D}^{\frac{3}{2}}_{t} u(t) =  \frac{-2e^{t}}{(1+e^{t})^{2}}u(t) - \frac{2 \omega e^{t}(1-e^{t})}{(1+e^{t})^{3}}\,\, ( := f(t, u(t))),
\end{align}
subject to the Dirichlet boundary conditions 
\begin{align}
\label{egbcs}
u(0) = 1,\,\,\,u(1) = 2.
\end{align}
Here $\omega$ is a scalar which in the context of the  flow of the Antarctic Circumpolar Currect corresponds to the dimensionless Coriolis parameter being equal to  $4649.56$. 

Let the BVP (\ref{egeq}), (\ref{egbcs}) be defined on the domain
\begin{align*}
    D := \{ u: 1 \leq u \leq 2\}, \,\,\,\,\, t \in [0, 1].
\end{align*}
Since $u : [0, 1] \rightarrow D \subset \mathbb{R}$, the constant vector  $M$ and matrices $K$ and $Q$, defined by (\ref{bounded}), (\ref{Lipschitz}), and (\ref{matrix}), respectively, are now scalars. We have
\begin{align*}
    M = 844.11, \,\,\,\,
    K = \frac{1}{2}, \,\,\,\, \beta = \frac{1}{3\sqrt{\pi}},\,\,\,\, Q = \frac{1}{6\sqrt{\pi}},
\end{align*}
thus, the condition of nonemptiness of the set $D_{\beta}$ is satisfied. Since ${Q< 1}$, $f(t, u(t))$ is bounded and satisfies a Lipschitz condition with constant $K$, conditions (\ref{bounded}) - (\ref{spectrad}) are satisfied. Hence, we can apply the numerical-analytic procedure derived in Sec. 2 - 4 to the present problem.

\medskip
\indent For the BVP (\ref{egeq}), (\ref{egbcs}), the approximate determining equation reads
\begin{align}
\label{apprdeteq}
 \Delta_{m}(\chi_{1})  = \frac{(1-\chi_{1})\sqrt{\pi}}{2}+\int_{0}^{1}(1-s)^{1/2}f(s, u_{m}(s, \chi_{1}))ds = 0,
\end{align}
and the sequence of approximations takes the form 

\begin{align}
\begin{split}
\label{um}
    u_{m}(t, \chi_{1}) =& 1 + \chi_{1}t + (1-\chi_{1})t^{3/2}+\frac{1}{\Gamma(3/2)} \int_{0}^{t}(t-s)^{1/2}f(s,u_{m-1}(s, \chi_{1}))ds \\
    &- \frac{1}{\Gamma(3/2)}t^{3/2}\int_{0}^{1}(1-s)^{1/2}f(s,u_{m-1}(s,\chi_{1}))ds,
    \end{split}
    \end{align}
    \begin{align}
    \label{u0}
    u_{0}(t, \chi_{1})  =& 1 + \chi_{1}t +(1-\chi_{1})t^{3/2},
\end{align}

\noindent where $m \in \mathbb{Z}^{+}$, $t \in [0, 1]$.\\

In order to obtain the approximate value of the parameter $\chi_{1} \in \Omega := [-333, -320]$, Eq. (\ref{apprdeteq}) is solved at each iteration step. At the initial step $m =0$ $u_{0}(t, \chi_{1})$, as given in (\ref{u0}), is substituted into the expression for $\Delta_{0}(\chi_{1})$, which yields
\begin{align*}
    \Delta_{0}(\chi_{1}^{0}) = \frac{(1-\chi_{1}^{0})\sqrt{\pi}}{2}+\int_{0}^{1}(1-s)^{1/2}f(s, u_{0}(s, \chi_{1}^{0}))ds,
\end{align*}
where 
\begin{align*}
    f(s, u_{0}(s, \chi_{1}^{0}))     =&     \frac{-2e^{s}[1+\chi_{1}^{0}s + (1-\chi_{1}^{0})(s-1)^{3/2}]}{(1+e^{s})^{2}}-\frac{2 \omega e^{s}(1-e^{s})}{(1+e^{s})^{3}}.
\end{align*}
The approximate determining equation
\begin{align*}
    \Delta_{0}(\chi_{1}^{0}) = 0
\end{align*}
is solved numerically to obtain $\chi_{1}^{0} = -320.68$. Thus, the initial approximation to the solution of BVP (\ref{egeq}), (\ref{egbcs}) is given by
\begin{align*}
    u_{0}(t, \chi_{1}^{0})=& 1 - 320.68t + 321.68t^{3/2}.
\end{align*}
At the next step, $m=1$, the expression for $u_{0}( t, \chi_{1})$ is used to construct the next approximation:
\begin{align*}
     u_{1}(t, \chi_{1}) =& 1 +\chi_{1}t + (1-\chi_{1})t^{3/2}  +\frac{1}{\Gamma(3/2)} \int_{0}^{t}(t-s)^{1/2}f(s,u_{0}(s, \chi_{1}))ds \\
      &- \frac{1}{\Gamma(3/2)}t^{3/2}\int_{0}^{1}(1-s)^{1/2}f(s,u_{0}(s,\chi_{1}))ds,\\
\end{align*}
which is substituted into $\Delta_{1}(\chi_{1})$:
\begin{align*}
    \Delta_{1}(\chi_{1}^{1})  = \frac{(1-\chi_{1}^{1})\sqrt{\pi}}{2}+\int_{0}^{1}(1-s)^{1/2}f(s, u_{1}(s, \chi_{1}^{1}))ds = 0.
\end{align*}
The approximate determining equation 
\begin{align*}
    \Delta_{1}(\chi_{1}^{1}) = 0
\end{align*}
is solved again to find $\chi_{1}^{1} = -332.06$. With the obtained value for $\chi_{1}^{1}$, the first approximation becomes
\begin{align*}
     u_{1}(t, \chi_{1}^{1}) =& 1 - 332.06t + 333.06t^{3/2}  +\frac{1}{\Gamma(3/2)} \int_{0}^{t}(t-s)^{1/2}f(s,u_{0}(s, \chi_{1}^{1}))ds \\
      &- \frac{1}{\Gamma(3/2)}t^{3/2}\int_{0}^{1}(1-s)^{1/2}f(s,u_{0}(s,\chi_{1}^{1}))ds,\\
\end{align*}
where 
\begin{align*}
    f(s, u_{0}(s, \chi_{1}^{1}))     =&     \frac{-2e^{s}[1-332.06s + 333.06(s-1)^{3/2}]}{(1+e^{s})^{2}}-\frac{2 \omega e^{s}(1-e^{s})}{(1+e^{s})^{3}}.
\end{align*}
Similarly, $u_{1}(t, \chi_{1})$ is used to construct $u_{2}(t, \chi_{1})$:
\begin{align*}
     u_{2}(t, \chi_{1}) =& 1 +\chi_{1}t + (1-\chi_{1})t^{3/2}  +\frac{1}{\Gamma(3/2)} \int_{0}^{t}(t-s)^{1/2}f(s,u_{1}(s, \chi_{1}))ds \\
      &- \frac{1}{\Gamma(3/2)}t^{3/2}\int_{0}^{1}(1-s)^{1/2}f(s,u_{1}(s,\chi_{1}))ds,\\
\end{align*}
which is substituted into $\Delta_{2}(\chi_{1})$ and the approximate determining equation $\Delta_{1}(\chi_{1}^{2}) = 0$ is solved to obtain $\chi_{1}^{2} = -332.30$. This value is substituted into the expression for $u_{2}(t, \chi_{1})$:
\begin{align*}
       u_{2}(t, \chi_{1}^{2}) =& 1 - 332.30t + 333.30t^{3/2}  +\frac{1}{\Gamma(3/2)} \int_{0}^{t}(t-s)^{1/2}f(s,u_{1}(s, \chi_{1}^{2}))ds \\
      &- \frac{1}{\Gamma(3/2)}t^{3/2}\int_{0}^{1}(1-s)^{1/2}f(s,u_{1}(s,\chi_{1}^{2}))ds.\\
\end{align*}

Figure~\ref{fig:1} shows plots of the first 3 approximations. In addition, we verified how well the calculated approximations satisfy the original FDE (\ref{fds}) by calculating the Caputo derivative of $u_{m}(t, \chi_{1}^{m})$ and comparing to the right-hand side $f(t, u_{m}(t, \chi_{1}^{m})) $ for $m = 0, 1, 2$. The plots are shown in Figures~\ref{fig:2}-\ref{fig:4}.
\begin{figure}[h!]
    \centering
    \includegraphics[width=0.9\linewidth]{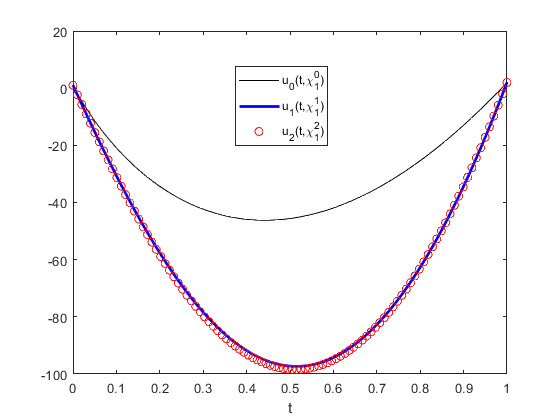}
    \caption{Numerical-analytic approximations to the solution of BVP (\ref{egeq}), (\ref{egbcs}) for $m =0,\,1,\,2$}
    \label{fig:1}
\end{figure}

\begin{figure}[h!]
\centering
        \includegraphics[width=0.92\linewidth]{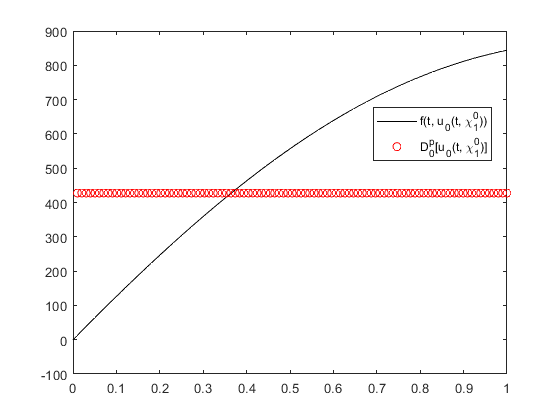}
        \caption{Right-hand side of BVP (\ref{egeq}), (\ref{egbcs}) (solid black line) and the approximations to the solution of the BVP  (solid red line) for $m=0$} \label{fig:2}
\end{figure}
\begin{figure}[h!]
\centering
        \includegraphics[width=0.92\linewidth]{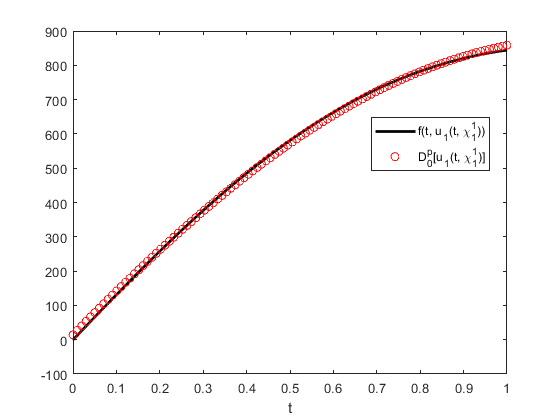}
        \caption{Right-hand side of BVP (\ref{egeq}), (\ref{egbcs}) (solid line) and the approximations to the solution of the BVP  (drawn with dots) for $m=1$} \label{fig:3}
\end{figure}
\begin{figure}[h!]
        \centering
        \includegraphics[width=0.92\linewidth]{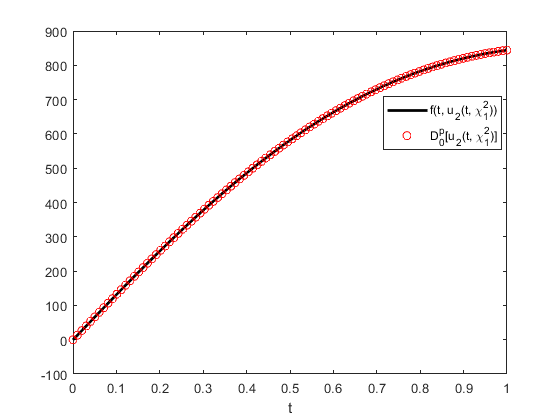}
        \caption{Right-hand side of BVP (\ref{egeq}), (\ref{egbcs}) (solid line) and the approximations to the solution of the BVP  (drawn with dots) for $m=2$} \label{fig:4}
\end{figure}

As it can be seen from our results, already on  the second iteration step we obtain a very good approximation to the exact solution of the original BVP \eqref{egeq}, \eqref{egbcs}. If necessary, this process can be continued even further and a better precision of computations can be obtained.  

\section{Conclusion}

\,\,\,\,\,\,\, Approximation methods are necessary for constructing approximate solutions to BVPs for which the exact solutions are not available. In this paper we use the numerical-analytic approximation technique to study a system of nonlinear FDEs of the Caputo type, subjected to the Dirichlet boundary conditions. We construct a sequence of functions and prove its uniform convergence to a limit function which is the exact solution to the IVP for the modified system of equations. We give necessary and sufficient conditions for the limit function to also satisfy the original BVP, and for the existence of solutions to the BVP.  \\
\indent The technique is applied to the equation modelling the motion of a gyre in the Southern hemisphere in the fractional setting. The approximate determining equation is solved numerically to obtain values of the unknown parameter, which are used to calculate the first three terms of the sequence. To verify the validity of the constructed approximations, we have checked how well they satisfy the original FDE.

The developed technique and existence results can be further extended and applied to more complex fractional BVPs.

\section*{Authors' contributions} 

Both authors have contributed equally to this paper.

\bibliographystyle{unsrt}
\bibliography{ConstApprox-MP-2022}


\end{document}